\documentclass[11pt,twoside]{amsart}
\usepackage{hyperref}
\usepackage[latin1]{inputenc}
\usepackage[OT1]{fontenc}
\usepackage{amsmath}
\usepackage{amsthm}
\usepackage{amssymb}
\usepackage[all]{xy}
\usepackage{ifthen} %HC
\usepackage{hyperref} %HC

\newtheorem{thm}{Theorem}[section]

\newtheorem{prop}[thm]{Proposition}

\newtheorem{lem}[thm]{Lemma}
\newtheorem{cor}[thm]{Corollary}
\newtheorem{conj}[thm]{Conjecture}

\numberwithin{equation}{section}

\theoremstyle{definition}
\newtheorem{definition}[thm]{Definition}
\newtheorem{remark}[thm]{Remark}

\newtheorem{notation}[thm]{Notation}

\newcommand{\IC}{\mathbb{C}}

\newcommand{\IP}{\mathbb{P}}

\newcommand{\IZ}{\mathbb{Z}}

\renewcommand{\to}{\xymatrix@1@=15pt{\ar[r]&}}
\renewcommand{\rightarrow}{\xymatrix@1@=15pt{\ar[r]&}}
\renewcommand{\mapsto}{\xymatrix@1@=15pt{\ar@{|->}[r]&}}
\renewcommand{\twoheadrightarrow}{\xymatrix@1@=15pt{\ar@{->>}[r]&}}
\renewcommand{\hookrightarrow}{\xymatrix@1@=15pt{\ar@{^(->}[r]&}}
\newcommand{\congpf}{\xymatrix@1@=15pt{\ar[r]^-\sim&}}

\begin{document}

\newboolean{xlabels} %HC
\newcommand{\xlabel}[1]{ %HC
                        \label{#1} %HC
                        \ifthenelse{\boolean{xlabels}} %HC
                                   {\marginpar[\hfill{\tiny #1}]{{\tiny #1}}} %HC
                                   {} %HC
                       } %HC
%\setboolean{xlabels}{true} %HC
\setboolean{xlabels}{false} %HC

\title[The Jordan-H\"older property]{On the Jordan-H\"older property for geometric derived categories}

\author[B\"ohning]{Christian B\"ohning$^1$}
\address{Christian B\"ohning, Fachbereich Mathematik der Universit\"at Hamburg\\
Bundesstra\ss e 55\\
20146 Hamburg, Germany}
\email{christian.boehning@math.uni-hamburg.de}

\author[Bothmer]{Hans-Christian Graf von Bothmer}
\address{Hans-Christian Graf von Bothmer, Fachbereich Mathematik der Universit\"at Hamburg\\
Bundesstra\ss e 55\\
20146 Hamburg, Germany}
\email{hans.christian.v.bothmer@uni-hamburg.de}

\author[Sosna]{Pawel Sosna$^2$}
\address{Pawel Sosna, Fachbereich Mathematik der Universit\"at Hamburg\\
Bundesstra\ss e 55\\
20146 Hamburg, Germany}
\email{pawel.sosna@math.uni-hamburg.de}

\thanks{$^1$ Supported by Heisenberg-Stipendium BO 3699/1-1 of the DFG (German Research Foundation)}
\thanks{$^2$ Supported by the RTG 1670 of the  DFG (German Research Foundation)}

\begin{abstract}
We prove that the semiorthogonal decompositions of the derived category of the classical Godeaux surface $X$ do not satisfy the Jordan-H\"older property. More precisely, there are two maximal exceptional sequences in this category, one of length $11$, the other of length $9$. Assuming the Noetherian property for semiorthogonal decompositions, one can define, following Kuznetsov, the Clemens-Griffiths component $\mathrm{CG}( \mathfrak{D})$ for each fixed maximal decomposition $\mathfrak{D}$. We then show that $\mathrm{D}^b (X)$ has two different maximal decompositions for which the Clemens-Griffiths components differ. Moreover, we produce examples of rational fourfolds whose derived categories also violate the Jordan-H\"older property. 
\end{abstract}

\maketitle

\section{Introduction}\xlabel{sIntroduction}

Over the past couple of decades the use of derived categories in algebraic geometry has become increasingly popular and successful. 
One advantage of this approach is that classical varieties are considered within a larger framework, namely \emph{noncommutative spaces} by which we mean triangulated $\mathbb{C}$-linear categories with a DG-enhancement. Roughly, whereas in the classical theory one considers things patched from commutative rings, here one passes to noncommutative $\mathbb{C}$-algebras, and more generally DG-algebras, or better, because everything should be ``derived Morita invariant", to the derived categories of modules over these DG-algebras. There are more maps between classical varieties when one considers them as noncommutative spaces (``Fourier-Mukai transforms"). One may also imagine  that, within the larger category of noncommutative spaces, unrelated commutative objects can possibly be connected by noncommutative deformations and this could shed more light on classical moduli problems.

The natural decompositions of a noncommutative space into simpler pieces are the semiorthogonal ones (semiorthogonality with respect to the $\mathrm{RHom}^{\bullet}$-pairing). See Section \ref{sSemiorthogonal} for details of the definition and background.

Semiorthogonal decompositions enjoy some good properties, for example, there is an interesting action of the braid group on their pieces and they behave well under birational modifications and other natural constructions like projective bundles. There are also properties that could be classified as a bit pathological, for example, it has been recognized recently \cite{BBKS12}, \cite{GorOrl} that certain pieces in semiorthogonal decompositions cannot be detected by any natural additive invariants (i.e., they evaluate to zero on them indiscriminately). These pieces were called phantoms because of this.

Some other foundational questions about semiorthogonal decompositions have remained open, however, in particular, whether the derived category of a variety (always smooth and projective here) has finite length with respect to these decompositions (the Noetherian property) and, more strongly, whether maximal such decompositions are always essentially unique (i.e.\ up to reordering of the pieces and equivalences of categories).  The latter is called the \emph{Jordan-H\"older property}. There are few references available in the literature, but it was discussed in \cite{Kuz09a}, \cite{KuzECM} and phrased as a question at the end of the article \cite{Kaw05}.

It is known that this fails for general triangulated categories by an example of Bondal, see \cite{KuzECM}: this is given as the derived category of representations of a certain quiver with relations. However, it is not of geometric origin, i.e.\ not equivalent to any $\mathrm{D}^b (Z)$ where $Z$ is a variety. Hence it was hoped that some sort of geometricity may save the Jordan-H\"older property in the context of derived categories of varieties. This would have very nice consequences: in \cite{Kuz10}, \cite{Kuz09a}, \cite{KuzECM} an approach to nonrationality of generic cubic fourfolds is sketched which is partly based on the Jordan-H\"older property and could be made into a complete proof if Jordan-H\"older were true.

In the present article we show that the Jordan-H\"older property fails also in the geometric set-up. Let $Y = \{ x_1^5 + x_2^5 + x_3^5 + x_4^5 =0 \}\subset \IP^3$ be the Fermat quintic with the $\IZ/5$-action given by $x_i \mapsto \xi^i x_i$ ($\xi$ a primitive fifth root of unity) and let $X= Y/(\IZ/5)$ be the so-called classical Godeaux surface, a surface of general type with $K^2=1$, $p_g=q=0$ and $\pi_1 = \IZ/5$.\smallskip

\noindent
{\textbf{Theorem \ref{tMain}}} \textit{The bounded derived category $\mathrm{D}^b(X)$ of coherent sheaves on the classical Godeaux surface $X$ does not satisfy the Jordan-H\"older property, namely it has two maximal exceptional sequences of different lengths: one of length $11$, which was already exhibited in \cite{BBS12}, and one of length $9$, which cannot be extended further.}
\smallskip

This failure of the Jordan-H\"older property is not confined to varieties of general type: in Corollary \ref{cFourfolds} we show that its failure on the Godeaux surface entails the failure on some rational fourfolds as well. Moreover, if one assumes the Noetherian property for semiorthogonal decompositions, the two exceptional sequences on $X$ give rise to two different maximal semiorthogonal decompositions whose respective Clemens-Griffiths components, see Definition \ref{dClemensGriffiths}, differ. 

The paper is organized as follows: Section \ref{sSemiorthogonal} assembles and recalls basic material concerning semiorthogonal decompositions and the Jordan-H\"older property. Section \ref{sLattice} explains the idea behind the construction of the length $9$ sequence which cannot be extended, and Section \ref{sCounter} carries out the details of this. 

The results of the present paper suggest that some new ideas will probably be needed before one may use derived category approaches to make progress on rationality problems, to mention just one major application. Possibly one has to  get a grip on all possible different semiorthogonal decompositions in a given situation and understand how the different ones are related. 

\textbf{Acknowledgment.} We would like to thank Sven Porst for letting us borrow his laptop after the one of the second author crashed. All computations for this article were done on his machine.

\section{Semiorthogonal decompositions and the Jordan-H\"older property}\xlabel{sSemiorthogonal}

We recall some basic notions, in particular, the Jordan-H\"older property and the definition of the Clemens-Griffiths component given in \cite{KuzECM}, \cite{Kuz09a}. 

We work over $k= \IC$. Let $\mathcal{T}$ denote a $k$-linear triangulated category. 

\begin{definition}\xlabel{dSemiorth}
A subcategory $\mathcal{S} \subset \mathcal{T}$ is called \emph{admissible} if the inclusion functor has a left and right adjoint. A sequence of admissible subcategories $(\mathcal{S}_1, \dots , \mathcal{S}_n)$ of $\mathcal{T}$ is called a \emph{semiorthogonal decomposition} of $\mathcal{T}$ if $\mathrm{Hom} (S_j, S_i) = 0$ for all objects $S_j \in \mathcal{S}_j$, $S_i \in \mathcal{S}_i$ with $j > i$, and, moreover, the subcategories $\mathcal{S}_i$ generate $\mathcal{T}$ in the sense that the smallest triangulated subcategory containing all of them is equivalent to $\mathcal{T}$. One writes
\[
\mathcal{T} = \langle \mathcal{S}_1, \dots , \mathcal{S}_n \rangle .
\]
\end{definition}

\begin{definition}\xlabel{dNoetherian}
The category $\mathcal{T}$ satisfies the \emph{Noetherian property} if every increasing sequence $\mathcal{S}_1 \subset \mathcal{S}_2\subset \dots $ of admissible subcategories becomes stationary. 
\end{definition}

Passing to the sequence of (right)orthogonals $^\perp\mathcal{S}_1 \supset ^\perp\mathcal{S}_2 \supset \dots $, one sees that the Noetherian property is equivalent to the Artinian property for $\mathcal{T}$, that is, every descending sequence of admissible subcategories becomes stationary. 

\begin{definition}\xlabel{dJordanHoelder}
A semiorthogonal decomposition $\mathcal{T} = \langle \mathcal{S}_1, \dots , \mathcal{S}_n \rangle$
is called \emph{maximal} if each $\mathcal{S}_i$ does not have any nontrivial semiorthogonal decompositions. The category $\mathcal{T}$ satisfies the \emph{Jordan-H\"older property} if, firstly, $\mathcal{T}$ satisfies the Noetherian property, and, secondly, if 
\[
\mathcal{T} = \langle \mathcal{S}_1, \dots , \mathcal{S}_n \rangle
\]
and
\[
\mathcal{T} = \langle \mathcal{U}_1, \dots , \mathcal{U}_m \rangle
\]
are two maximal semiorthogonal decompositions, then $m=n$ and there exists a bijection $\sigma\colon \{ 1, \dots , n\} \to \{ 1, \dots , m \}$ such that $\mathcal{S}_i$ is equivalent to $\mathcal{U}_{\sigma (i)}$.  
\end{definition}

\begin{definition}\xlabel{dEssentialDimension}
The \emph{essential dimension} $\mathrm{ess.dim} (\mathcal{T})$ of $\mathcal{T}$ is the minimum integer $n$ such that there exists a fully faithful embedding $\mathcal{T} \hookrightarrow \mathrm{D}^b (X)$ with admissible image, where $X$ is a smooth projective variety of dimension $n$. We set $\mathrm{ess.dim}(\mathcal{T} ) = \infty$ if no such integer exists. 
\end{definition}

\begin{definition}\xlabel{dClemensGriffiths}
Let $\mathfrak{D}$ be a maximal semiorthogonal decomposition of $\mathcal{T}= \mathrm{D}^b (X)$ where $X$ is a smooth projective variety of dimension $n$:
\[
\mathcal{T} = \langle \mathcal{S}_1, \dots , \mathcal{S}_n \rangle .
\]
The \emph{Clemens-Griffiths component} $\mathrm{CG}(\mathfrak{D})$ of $\mathfrak{D}$ is defined as the smallest triangulated subcategory of $\mathcal{T}$ which contains all those $\mathcal{S}_i$ with $\mathrm{ess.dim}(\mathcal{S}_i) \ge n-1$. 
\end{definition}

We will construct examples of varieties $X$ where $\mathrm{D}^b(X)$ does not satisfy the Jordan-H\"older property. We can take for $X$ the classical Godeaux surface. In fact, we will construct maximal exceptional sequences of different lengths. Assuming the Noetherian property for $\mathrm{D}^b(X)$ this implies that there are two different maximal decompositions whose Clemens-Griffiths components are not equivalent. Recall the

\begin{definition}\xlabel{dExceptional}
An object $E \in \mathrm{D}^b (X)$ is called \emph{exceptional} if $\mathrm{RHom}^{\bullet}(E, E) \simeq \IC$. An \emph{exceptional sequence} is a sequence of exceptional objects $(E_1, \dots ,E_m)$ with $\mathrm{RHom}^{\bullet}(E_j, E_i)=0$ for $j>i$. The sequence is called \emph{unextendable} if there does not exist an exceptional object $F$ in $\mathrm{D}^b(X)$ such that $(E_1, \dots , E_m, F)$ is an exceptional sequence. \end{definition}

If an exceptional sequence is unextendable, there is no exceptional object $F'$ such that $(F', E_1, \dots , E_m)$ is an exceptional sequence either: otherwise one can mutate $F'$ to the right.

Every exceptional sequence generates an admissible subcategory, and an admissible subcategory is generated by an exceptional object if and only if it is equivalent to the derived category of finite dimensional $k$-vector spaces.

Now we briefly discuss the impact of this on derived category approaches to rationality of varieties: in \cite{Kuz10}, Kuznetsov has shown that for a smooth cubic fourfold $V$ there is a semiorthogonal decomposition
\[
\mathrm{D}^b (V) = \langle \mathcal{A}_Y, \mathcal{O}_V, \mathcal{O}_V (1), \mathcal{O}_V (2) \rangle
\]
and, for many classes of rational cubic fourfolds, $\mathcal{A}_Y$ is equivalent to the derived category of a K3 surface, and, moreover, there are cases where $\mathcal{A}_Y$ is a derived category of a ``noncommutative" K3 surface (a $2$-Calabi-Yau category) of essential dimension bigger than $2$. In these cases one does not expect $V$ to be rational.

Assuming the Jordan-H\"older property for derived categories of smooth projective varieties, this would imply nonrationality of the generic cubic fourfold since one could then speak of \emph{the} Clemens-Griffiths component of a variety and this would be birationally invariant (the latter would follow immediately from the structure of derived categories of blow-ups, Beilinson's theorem for $\mathrm{D}^b (\mathbb{P}^n)$ and the factorization theorem for birational maps into a sequence of blow-ups and blow-downs in smooth centers). 

Our results suggest that this approach needs some substantial modification. However, it is very beautiful, and one may hope that either the Clemens-Griffiths component is, by some miracle or deeper principle, a well-defined notion for fourfolds of Kodaira dimension $-\infty$, or one may control \emph{all} semiorthogonal decompositions of $\mathrm{D}^b(V)$ (or $\mathrm{D}^b(\mathbb{P}^4)$) suitably, as well as their behaviour under elementary birational transformations.

\section{Idea of the construction}\xlabel{sLattice}

For an arbitrary smooth projective variety $Z$, we have a nonsymmetric bilinear pairing
\[
\chi (\cdot , \cdot ) : \mathrm{K}_0 (Z) /(\mathrm{tors}) \times \mathrm{K}_0 (Z) / (\mathrm{tors}) \to \IZ
\]
given by the Euler form. 

\begin{definition}
A sequence of objects $(e_1, \dots , e_n)$ is called \emph{numerically semiorthonormal} if $\chi (e_i, e_i ) =1$ and $\chi (e_j, e_i ) = 0$ whenever $j > i$. It is called a \emph{numerically semiorthonormal basis} if moreover $e_1, \dots , e_n$ is a basis of $\mathrm{K}_0 (Z) /(\mathrm{tors})$. This is equivalent to $n = \mathrm{rk} (\mathrm{K}_0 (Z) /(\mathrm{tors}))$. 
\end{definition}

Clearly, the classes in $\mathrm{K}_0(Z)/(\mathrm{tors})$ of the objects in an exceptional sequence $(E_1, \dots , E_n)$ in $\mathrm{D}^b (Z)$ form a numerically semiorthonormal sequence.

The idea of the construction of our unextendable exceptional sequence is based on the existence of an isomorphism of lattices $\mathrm{K}_0 (X)/(\mathrm{tors }) \simeq \mathrm{K}_0 (S)$ where $S$ is a del Pezzo surface of degree $1$, i.e.\ the blow-up of $\mathbb{P}^2$ in eight points, and the fact that there are line bundles $\mathcal{L}$ on $S$ with $\chi (\mathcal{L} ) = 0$, but $\mathrm{R}\Gamma^{\bullet } (\mathcal{L}) \neq 0$, which under the above isomorphism correspond to classes in $\mathrm{K}_0 (X) /(\mathrm{tors})$ which can be realized by acyclic line bundles $\overline{\mathcal{L}}$ on $X$.  Here acyclic means $\mathrm{R}\Gamma^{\bullet} (\overline{\mathcal{L}})=0$. That is, the idea is to exploit the \emph{difference of the notion of effectivity} on $S$ and $X$.

There seems to be the folklore result that all numerically semiorthonormal bases of $\mathrm{K}_0 (S)$ form one braid group orbit up to tensoring by classes of line bundles and sign changes of basis vectors. However, whereas this is proven for complete exceptional sequences on $S$, we could not find a proof for the lattice theoretic result in the literature. However, we used it as a heuristic philosophy to find the counterexample. This is formalized in the following 

\begin{lem}\xlabel{lHeuristic}
If there are line bundles $\mathcal{L}$ on $S$ and $\overline{\mathcal{L}}$ on $X$ with the above properties and if numerically semiorthonormal bases in $\mathrm{K}_0 (S)$ form one braid group orbit up to tensoring by classes of line bundles and sign changes of basis vectors, then the Jordan-H\"older property does not hold for $\mathrm{D}^b (X)$. 
\end{lem}

\begin{proof}
In the above set-up we have an exceptional pair $(\mathcal{O}_X, \overline{\mathcal{L}}^{-1})$ in $\mathrm{D}^b (X)$. We claim that it cannot be extended to a numerically semiorthonormal basis even in $\mathrm{K}_0 (X)/(\mathrm{tors})$. In fact, otherwise, $(\mathcal{O}_S, \mathcal{L}^{-1})$ can also be extended to a semiorthonormal basis in $\mathrm{K}_0 (S)$. By assumption and because there are full exceptional sequences in $\mathrm{D}^b (S)$, we get that there is a full exceptional sequence $(E_1, \dots , E_{11})$ in $\mathrm{D}^b (S)$ such that $[E_1] = [ \mathcal{O}_S]$ and $[E_2] = [\mathcal{L}^{-1}]$ in $\mathrm{K}_0 (S)$. But exceptional objects in $\mathrm{D}^b (S)$ are pure, i.e.\ shifts of sheaves, and these sheaves are either locally free or supported on exceptional curves. Hence we must have that $E_1$ and $E_2$ are actually isomorphic to $\mathcal{O}_S$ and $\mathcal{L}^{-1}$ (up to a shift). But then $(\mathcal{O}_S, \mathcal{L}^{-1})$ would be an exceptional pair, a contradiction because $\mathcal{L}$ is not acyclic.

Hence $(\mathcal{O}_X, \overline{\mathcal{L}}^{-1})$ is not extendable to an exceptional sequence of length $11$ in $\mathrm{D}^b (X)$. Since there is an exceptional sequence of length $11$ in $\mathrm{D}^b (X)$ by \cite{BBS12}, the Jordan-H\"older property does not hold.
\end{proof}

Thus we first find candidates for $\mathcal{L}$ and $\overline{\mathcal{L}}$ above. We need to introduce some notation for this:

\begin{notation}\xlabel{nDelPezzoGodeaux}
On the del Pezzo surface $S$ we write $K_S$ for the canonical class, $h$ for the pull-back of the hyperplane class, $e_1, \dots , e_8$ for the eight exceptional curves of the blow-up $S \to \mathbb{P}^2$. Then $\mathrm{Pic} (S)$ is generated by $h, e_1, \dots , e_8$ and the intersection matrix with respect to this basis is  
$\mathrm{diag} (1, -1, \dots , -1)$. One has $K_S = - 3h + \sum_i e_i$. As a lattice $\mathrm{Pic} (S) \simeq \mathbf{1} \perp (-E_8)$ where $\mathbf{1}$ is generated by $K_S$. 

One also has $\mathrm{Pic}(X)/(\mathrm{tors}) \simeq \mathrm{Pic}(S)$ as lattices for $X$ the Godeaux surface. On $X$ there are precisely $50$ (smooth) elliptic curves of (canonical) degree $1$ which were made explicit in \cite{BBS12}. We denote them by $E_{i,j}^{\pm}$, $i, j \in \mathbb{Z}/5$, as in that paper. Moreover, we denote the canonical class by $K_X$ and abbreviate $E_i^{\pm} := E^{\pm}_{i,0}$. There is a pencil of genus $2$ curves on $X$ with five reducible fibers consisting of the two elliptic curves $E_i^+$ and $E_i^-$ meeting transversely.
\end{notation}

We can take $\mathcal{L}= \mathcal{O}_S (2e_1)$. Then $\chi (\mathcal{L})=0$, but certainly $\mathcal{L}$ has a section, hence is not acyclic. In fact, we may write $\mathcal{O}_S(e_1) = -K_S + R$, whence $R$ is a root, i.e.\ $R^2= -2$, $K_S.R = 0$. Then there is an essential difference between $S$ and $X$:

\begin{prop}\xlabel{pEffectivity}
The bundle \[ \mathcal{L} = \mathcal{O}_S (2 (-K_S +R))\] is never acyclic, for any root $R$ in $\mathrm{Pic}(S)$. However, on the Godeaux surface $X$, the line bundle
\[
\overline{\mathcal{L}} = \mathcal{O}_X (2 (-K_X +R'))
\]
can be both acyclic or not acyclic depending on the choice of root $R'$. It is acyclic, for example, if $R' = E_{1}^+ - E_{2}^+$, and not acyclic if $R' = K_X - E_1^+$. 
\end{prop}

\begin{proof}

There are $240$ roots $R$ in the $(-E_8)$-lattice. Now $V= -K_S + R$ is an exceptional vector in the terminology of \cite[Ch.\ 8]{Dolg}: it satisfies
\[
V. K_S = -1, \quad V^2 = -1 \, .
\]
There are $240$ of these on a del Pezzo surface of degree $1$ and they correspond precisely to the $(-1)$-curves on $S$. In particular, they are all effective by \cite[Lem.\ 8.2.22]{Dolg}.

Let us now show the second part of Proposition \ref{pEffectivity}. 
Taking $R' = K_X - E_1^+$ one gets
\[
\overline{\mathcal{L}} = \mathcal{O}_X (-2E_1^+)\, .
\]
However, $h^2 (\overline{\mathcal{L}} ) = h^0 ( K_X + 2E_1^+ )$ and the exact sequence
\[
0 \to \mathcal{O}_X (K_X + E_1^+ ) \to \mathcal{O}_X ( K_X + 2E_1^+) \to \mathcal{O}_{E_1^+} ( K_X + 2E_1^+) \to 0
\]
yields $h^0 (\mathcal{O}_X (K_X + E_1^+ )) = h^0 (\mathcal{O}_X ( K_X + 2E_1^+))$ since $(K_X + 2E_1^+).E_1^+ = -1$. But then the exact sequence
\[
0 \to \mathcal{O}_X (K_X) \to \mathcal{O}_X (K_X + E_1^+ ) \to \omega_{E_1} \to 0
\]
yields $h^0 (\mathcal{O}_X (K_X + E_1^+ ) ) = 1$. 

Now look at the case $R' = E_1^+ - E_2^+$. Clearly, for degree reasons, $\overline{\mathcal{L}}$ has no sections in this case, and by Serre duality, we just have to show that
\[
h^0 (3 K_X - 2R' ) = h^0 (3K_X - 2 E_1^+ + 2 E_2^+ )
\]
is zero. This follows from Lemma \ref{lEffective} below and the fact that $K_X - R'$ is not effective by \cite[Cor.\ 6.4(2)]{BBS12}.
\end{proof}

\begin{lem}\xlabel{lEffective}
Let $R'$ be a root in $\mathrm{Pic}(X)$ such that $h^0 (X, K_X-R') = 0$. Then 
\[
h^0 (3K_X -2R') =0.
\]
\end{lem}

\begin{proof}
Let $B= R'+ 2K$. Then $\chi(B)=1$ and $h^2(B)=h^0(K-B)=0$ for degree reasons. Hence $B$ is effective. Then 
\[ 3K - 2R' = 7K- 2B.\] 
Since 
\[
h^0(3K -B) = h^0(K-R') = 0
\]
the ideal $I_{B, X}$ is generated in degrees $\ge 4$. Therefore, the ideal $I_{B,X}^2$ is generated in degrees $\ge 8$. By \cite[Thm.\ A]{ELS} the second symbolic power of $I_{B, X}$ is equal to $I_{B, X}^2$ and hence we get $h^0(7K - 2B)=0$. 
\end{proof}

\begin{remark}\xlabel{rAlternative}
Alternatively, one can check that $h^0 (3K-2R') = 0$ as follows.
By \cite{BBS12} there exists a line $L^0_0$ on $X$ such that $3K_X - L_0^0$ gives the genus $2$ pencil on $X$ and hence we can write
\begin{gather*}
3K_X - 2E_1^+ + 2E_2^+ = 3K_X - 2E_1^+ + 2(3K_X - L_0^0 -E_2^-)\\
= 9K_X - 2L_0^0 - 2E_1^+ - 2E_2^- .
\end{gather*}
We then look at the cover $p\colon Y \to X$, where $Y \subset \IP^3$ is the Fermat quintic with its $\IZ/5$-action, and remark that there are no $\IZ/5$-invariant sections in $H^0 (Y, 9K_Y) = H^0 (Y, \mathcal{O}(9))$ which vanish in the subscheme $p^* ( 2L_0^0 + 2E_1^+ + 2E_2^- )$ of $Y$. This can be checked by a computer algebra calculation with Macaulay 2 \cite{BBS12a}. In fact, there are even no degree $9$ polynomials on $\IP^3$ in the ideal of the subscheme $ 2L_0^0 + 2E_1^+ + 2E_2^-$ except multiples of the Fermat equation. 

\end{remark}

Hence our expectation is

\begin{conj}\xlabel{cExceptionalPair}
The exceptional pair 
\[
\Bigl( \mathcal{O}_X, \mathcal{O}_X \bigl( 2(K_X - E_1^+ + E_2^+)\bigr) \Bigr)
\]
cannot be extended to a numerically  semiorthonormal basis in $\mathrm{K}_0 (X) /(\mathrm{tors})$. 
\end{conj}

As we explained above, this would follow from the expected transitivity result for mutations on numerically semiorthonormal bases. We checked Conjecture \ref{cExceptionalPair} probabilistically by a computer, i.e. we found no numerical extensions. In fact, the results of that experiment were that the maximal length to which the above exceptional pair could be extended numerically is $9$. All extensions we found were of  the numerical shape
\[
(\mathcal{O}_S, \mathcal{O}_S (-2e_1), \mathcal{O}_{C_1}, \dots , \mathcal{O}_{C_7})
\]
with the $\mathcal{O}_{C_i}$ structure sheaves of $(-1)$-curves $C_i$ on $S$ such that $(\mathcal{O}_{C_1}, \dots , \mathcal{O}_{C_7})$ is completely orthogonal in $\mathrm{K}_0 (S)$. 

We circumvent this difficulty by exhibiting an unextendable exceptional sequence  of length $9$ 
%Start with the numerically semiorthonormal sequence
%\[
%(\mathcal{O}_S, \mathcal{O}_S(-2e_1), \mathcal{O}_{e_2}, \dots \mathcal{O}_{e_8}) .
%\]
%Mutate the $\mathcal{O}_{e_i}$ to the left across $ \mathcal{O}_S(-2e_1)$ and twist by $\mathcal{O}_S(K_S)$ to obtain
\[
(\mathcal{O}_S(K_S) , \mathcal{O}_S (K_S + e_2), \dots , \mathcal{O}_S (K_S +e_8), \mathcal{O}_S (K_S - 2e_1)) =: (m_1, \dots , m_9).
\]
which contains $\mathcal{O}_S (2e_1)$ as a difference $m_9^{\vee}\otimes  m_1$. 
Notice that $\mathcal{O}_S (K_S + e_i)$, $i=1, \dots , 8$ is a numerically orthogonal set of roots. 

\begin{remark}\xlabel{rNumerics}
The sequence $(m_1, \dots , m_9)$ is related to the sequence
\[
(\mathcal{O}_S, \mathcal{O}_S(-2e_1), \mathcal{O}_{e_2}, \dots \mathcal{O}_{e_8}) .
\]
as follows. Consider the sequence of (derived) duals
\[
(\mathcal{O}_{e_8}^{\vee}, \dots , \mathcal{O}_{e_2}^{\vee}, \mathcal{O}_S(2e_1), \mathcal{O}_S)
\]
and mutate the $\mathcal{O}_{e_i}^{\vee}$ to the right across $\mathcal{O}_S(2e_1)$ to obtain 
\[
(\mathcal{O}_S (2e_1), \mathcal{O}_S(2e_1 + e_8)[1], \dots , \mathcal{O}_S(2e_1 + e_2)[1], \mathcal{O}_S).
\]
Forgetting the shifts, twisting by $\mathcal{O}_S(K-2e_1)$ and reordering the completely orthogonal terms,  we get the above sequence.
\end{remark}

\begin{definition}\xlabel{dLattice}
Consider the lattice 
\[
\Lambda:= \left\{ x   + y_0  h + y_1 e_1 + \dots + y_8 e_8 + \frac{1}{2} z p \right\}\subset \mathrm{CH}^* (S) [1/2]
\]
where $(x, y_0, y_1, \dots , y_8, z)\in \IZ^{11}$ and $p$ is the class of a point. We set $v= (x,y,z)$, where
\[
y = y_0  h + y_1 e_1 + \dots + y_8 e_8.
\]
\end{definition}

We have a Chern character map 
\[
\mathrm{ch}\colon \mathrm{K}_0 (S) \to \mathrm{CH}^* (S) [1/2]
\]
sending a vector bundle $\mathcal{E}$ of rank $r$ to
\[
\mathrm{ch}(\mathcal{E}) = r + c_1 (\mathcal{E}) +  \frac{1}{2} (c_1(\mathcal{E})^2 - 2 c_2 (\mathcal{E})) .
\]
Notice that the Chern character is injective. It identifies the lattice $\mathrm{K}_0 (S)$ with a lattice $\IZ^{11}\simeq \mathrm{ch}(\mathrm{K}_0(S))\subset \Lambda \subset \mathrm{CH}^*(S)[1/2]$ with basis
\[
1, h + \frac{1}{2}p, e_1 - \frac{1}{2}p, \dots , e_8 - \frac{1}{2}p , p.
\]
This basis is obtained for example from the exceptional sequence \[ (\mathcal{O}_S, \mathcal{O}_S (h), \mathcal{O}_S (2h), \mathcal{O}_{e_1}, \dots , \mathcal{O}_{e_8}).\] 
The lattice $\Lambda$ contains the previous lattice as a sublattice of index $2$. 

The Riemann-Roch theorem says
\[
\chi (S, \mathcal{E}) = \deg \left( \mathrm{ch}(\mathcal{E}).\mathrm{td}(\mathcal{T}_S)\right)_2
\]
where
\[
\mathrm{td}(\mathcal{T}_S) = 1 - \frac{1}{2}K_S + \frac{1}{12} (K_S^2 + c_2 ) = 1 - \frac{1}{2} K_S + p.
\]
The subscript $2$ in the second but last formula means that one only considers the top dimensional component. Hence in terms of the vector $v= (x,y,z)$
\[
\chi (S, \mathcal{E}) = x - \frac{1}{2} y . K_S + \frac{1}{2} z \, . 
\]
If $\mathcal{E}_1$ and $\mathcal{E}_2$ are two bundles, then 
\[
\chi (\mathcal{E}_1, \mathcal{E}_2 ) = \chi (S, \mathcal{E}_1^{\vee} \otimes \mathcal{E}_2 )
\]
and 
\begin{gather*}
\mathrm{ch}(\mathcal{E}_1^{\vee} \otimes \mathcal{E}_2 ) = \mathrm{ch}(\mathcal{E}_1^{\vee}). \mathrm{ch} (\mathcal{E}_2) = (x_1 - y_1 + \frac{1}{2} z_1 )( x_2 + y_2 + \frac{1}{2} z_2) \\
= x_1x_2  + (x_1y_2 - x_2 y_1) + \frac{1}{2}(x_1z_2 + x_2z_1 - 2y_1y_2)
\end{gather*}
whence
\[
\chi (\mathcal{E}_1, \mathcal{E}_2 ) = x_1x_2 - \frac{1}{2} (x_1y_2 - x_2 y_1).K_S + \frac{1}{2} (x_1z_2 + x_2z_1-2y_1y_2) \, .
\]
Clearly, since every sheaf has a resolution by locally free ones on $S$ and both sides of the previous equation are bilinear in $\mathcal{E}_1$, $\mathcal{E}_2$ resp. $v_1$, $v_2$, the formula holds for arbitrary classes $\mathcal{E}_1$ and $\mathcal{E}_2$ in $\mathrm{K}_0 (S)$. Also, $x,y,z$ have integer coordinates for arbitrary classes in $\mathrm{K}_0 (S)$.

\begin{prop}\xlabel{pUnextendable}
The sequence $(m_1, \dots , m_9)$ is numerically unextendable in $\mathrm{K}_0 (S)$. 
\end{prop}

\begin{proof}

We consider
\begin{align*}
	m_{10} &= -h + \frac{3}{2}p \\
	m_{11} &=  2 + 2k + h - 3e_1 - p 
\end{align*}
in the lattice $\mathrm{ch}(\mathrm{K}_0 (S)) \subset \Lambda$. One checks by direct computation that $m_{10}$ and $m_{11}$ are numerically semiorthogonal to $m_1,\dots, m_9$ and span the orthogonal complement over $\IZ$ in the lattice $\mathrm{ch}(\mathrm{K}_0(S))$. Indeed, the matrix
\[
	\begin{pmatrix}
		\chi(m_{10},m_{10}) & \chi(m_{10},m_{11}) \\
		\chi(m_{11},m_{10}) & \chi(m_{11},m_{11}) 
	\end{pmatrix}
		=
	\begin{pmatrix}
		-1 &	1 \\
		-5 &	4
	\end{pmatrix}
\]
hat determinant equal to $1$. Furthermore,
\[
	\chi(sm_{10}+tm_{11}) = -s^2-4st+4t^2.
\]
Since $-s^2-4st+4t^2 = 1$ has no solution modulo $4$, there exists no class in $\mathrm{K}_0(S)$ that is numerically semiorthogonal to $m_1,\dots,m_9$ and numerically exceptional. 
\end{proof}

\section{The counterexamples}\xlabel{sCounter} 

Before proceeding further we recall the following fact from \cite{BBS12}: the roots

\begin{gather*}
A_1 = E^-_{0,4}-E^+_{4,4}, \\
%\alpha_2=  -(E^+_{0,0} - E^+_{1,0} -E^+_{2,0} +E^+_{4, 0} - (E^+_{4,4} - E^-_{0,4})), \\
A_2 = E_{4,0}^+ -E_{3,0}^+, \\
A_3 = E_{3,0}^+ -E_{2,0}^+, \\
A_4 = E_{2,0}^+ -E_{1,0}^+, \\
A_5 = E_{1,0}^+ - E_{0,0}^-, \\
%\alpha_7 = E_{4,1}^+ -E_{0,1}^- - (E_{4,0}^+ -E_{0,0}^-), \\
%\alpha_8 = E_{4,2}^+ -E_{0,2}^- - (E_{4,1}^+ - E_{0,1}^-)\\
A_6 = E_{0,2}^- - E_{0,4}^-, \\
A_7 = E_{0,3}^- - E_{0,0}^- , \\
A_8 = E^-_{0,4} - E_{0,1}^- 
\end{gather*}

in $\mathrm{Pic}(X)/(\mathrm{tors})$ form a $(-A_8)$-subsystem of the root lattice $(-E_8) \subset \mathrm{Pic}(X)/(\mathrm{tors})$. The bundles 
\[
\mathcal{O}(A_1),  \mathcal{O}(A_1+A_2), \dots , \mathcal{O}(A_1 + \dots + A_8)
\]
form a completely orthogonal exceptional sequence whose terms give roots in $\mathrm{Pic}(X)/(\mathrm{tors})$ by \cite{BBS12}. 

\begin{prop}\xlabel{pSequence}
Consider the following line bundles on $X$:
\begin{eqnarray*}
\mathcal{M}_1 &= \mathcal{O}( K_X)\\
\mathcal{M}_2 &= \mathcal{O}(A_1+A_2)\\
\mathcal{M}_3 &= \mathcal{O}(A_1+A_2+A_3)\\
\mathcal{M}_4 &= \mathcal{O}(A_1+\dots +A_4) \\
\mathcal{M}_5 &= \mathcal{O}(A_1+\dots +A_5)\\
\mathcal{M}_6 &= \mathcal{O}(A_1+\dots +A_6)\\
\mathcal{M}_7 &= \mathcal{O}(A_1+\dots +A_7)\\
\mathcal{M}_8 &=\mathcal{O}(A_1+\dots +A_8) \\
\mathcal{M}_9 &= \mathcal{O}(3K -2A_1) .
\end{eqnarray*}
Then $(\mathcal{M}_1, \dots , \mathcal{M}_9)$ becomes an exceptional sequence after tensoring each of the $\mathcal{M}_i$ with an appropriate torsion line bundle. 

\end{prop}

\begin{proof}
It follows from \cite{BBS12} that the sequence $(\mathcal{M}_1, \dots , \mathcal{M}_8)$ is exceptional and this remains true after twisting all these bundles by arbitrary torsion line bundles.  To prove that there exists a torsion line bundle $\mathcal{O}_{\tau}$ such that  $\mathcal{M}_i \otimes \mathcal{M}_9^{-1}\otimes \mathcal{O}_{\tau}$ is acyclic, we write up to numerical equivalence
\[
\mathcal{M}_i - \mathcal{M}_9 \sim nK_X - D
\]
with $D$ effective and check that the ideal of $p^* (D)\subset Y$ has less than five sections in degree $n$. A Macaulay 2 script doing this can be found at \cite{BBS12a}. 
\end{proof}

\begin{prop}\xlabel{pQuadric}
The sequence of Proposition \ref{pSequence} is unextendable. That is, there does not exist an exceptional object $\mathcal{F} \in \mathrm{D}^b (X)$ such that the sequence $(\mathcal{M}_1,\dots ,\mathcal{M}_9,\mathcal{F})$ is exceptional. 
\end{prop}

\begin{proof}
Notice that there is an isometry of lattices $\mathrm{Pic}(S) \to \mathrm{Pic}(X)/(\mathrm{tors})$ mapping $(m_1, \dots , m_9)$ to $(\mathcal{M}_1, \dots , \mathcal{M}_9)$. This follows since each orthogonal system of eight roots is of the form
\[
A_1, A_1+A_2, \dots , A_1 + \dots + A_8
\]
where $A_1, \dots , A_8$ form a $(-A_8)$-subsystem. The corresponding $(-A_8)$-system in $\mathrm{Pic}(S)$ is $(k+e_1, e_2-e_1, \dots , e_8-e_7)$, and the Weyl group action on $(-A_8)$ subsystems of the $(-E_8)$-lattice is transitive by the Borel-Siebenthal theorem. 
By Proposition \ref{pUnextendable}, this immediately implies the result.
\end{proof}

\begin{thm}\xlabel{tMain}
The derived category of the classical Godeaux surface $X$ does not satisfy the Jordan-H\"older property. 
\end{thm}

\begin{proof}
This is clear from Propositions \ref{pSequence} and \ref{pQuadric} together with the fact that $\mathrm{D}^b (X)$ has an exceptional sequence of length $11$ by \cite{BBS12}. 
\end{proof}

\begin{cor}\xlabel{cClemensGriffiths}
Assume that the Noetherian property holds for $\mathrm{D}^b (X)$. Then there exists two maximal semiorthogonal decompositions $\mathfrak{D}_1$ and $\mathfrak{D}_2$ on $X$ whose Clemens-Griffiths components $\mathrm{CG}(\mathfrak{D}_1)$ and $\mathrm{CG}(\mathfrak{D}_2)$ are not equivalent.
\end{cor}

\begin{proof}
We have two semiorthogonal decompositions
\begin{gather*}
\mathrm{D}^b (X) = \langle \mathcal{L}_1, \dots , \mathcal{L}_{11}, \mathcal{B}_1 \rangle , \\
\mathrm{D}^b (X) = \langle \mathcal{M}_1, \dots , \mathcal{M}_9, \mathcal{B}_2 \rangle 
\end{gather*}
where $(\mathcal{L}_1, \dots , \mathcal{L}_{11})$ is the exceptional sequence from \cite{BBS12}, and $\mathcal{B}_1$ and $\mathcal{B}_2$ are the respective orthogonal complements. Here $\mathcal{B}_1$ is a quasi-phantom with Grothendieck group $\IZ/5$. We can go on decomposing $\mathcal{B}_1$ and $\mathcal{B}_2$ if possible until we reach two maximal decompositions $\mathfrak{D}_1$ and $\mathfrak{D}_2$. Note that an indecomposable piece in a semiorthogonal decomposition has essential dimension $0$ if and only if it is generated by an exceptional object, and $\mathrm{CG}(\mathfrak{D}_i)$ are hence obtained by grouping together all those indecomposable pieces which are not generated by an exceptional object. Hence $\mathcal{B}_1$ and $\mathcal{B}_2$ are already equal to $\mathrm{CG}(\mathfrak{D}_1)$ resp.\ $\mathrm{CG}(\mathfrak{D}_2)$ in this case. But their Grothendieck groups have different ranks, so they are not equivalent.
\end{proof}

\begin{cor}\xlabel{cFourfolds}
There exist rational smooth projective fourfolds $Z$ for which $\mathrm{D}^b (Z)$ does not satisfy the Jordan-H\"older property.
\end{cor}

\begin{proof}
We can embed the surface $X$ into $\IP^5$ and find a generic projection to $\IP^4$ such that the image $\bar{X} \subset \IP^4$ has improper double points as only singularities (which look like two planes meeting transversally in one point locally). Blowing up the double points, we have an embedding $X \subset \tilde{\mathbb{P}}^4$. Now consider the fourfold $Z = \mathrm{Bl}_{X} (\tilde{\mathbb{P}}^4)$. By a result of \cite{Orlov93} the derived category of $Z$ has a semiorthogonal decomposition into a copy of $\mathrm{D}^b (\tilde{\mathbb{P}}^4)$ and one copy of the blow-up center $\mathrm{D}^b (X)$. Now we can produce an exceptional sequence in $\mathrm{D}^b (Z)$ by concatenating a full exceptional sequence in the copy $\mathrm{D}^b (\tilde{\mathbb{P}}^4)$ with  the sequence $(\mathcal{M}_1, \dots , \mathcal{M}_{9})$ above in $\mathrm{D}^b (X)$. For the same numerical reasons as above, there is no exceptional object in the (left) orthogonal to this sequence. However, there are exceptional sequences of greater length in $\mathrm{D}^b (Z)$ (simply choose $(\mathcal{L}_1, \dots , \mathcal{L}_{11})$ also in $\mathrm{D}^b (X)$). This proves the Corollary. 
\end{proof}

The last Corollary shows that the failure of the Jordan-H\"older property is clearly not restricted to manifolds of general type, and it even fails for varieties which one has to control when trying to implement the approach to nonrationality of generic cubic fourfolds suggested in \cite{Kuz10}. Thus, as suggested at the end of Section \ref{sSemiorthogonal}, probably a substantial modification or, in any case, further ideas will be needed here.

\end{document}